\documentclass[11pt,a4paper]{amsart}
\usepackage[USenglish]{babel}
\usepackage[utf8]{inputenc}
\pdfoutput=1
\usepackage{amssymb}
\usepackage{graphicx}
\usepackage{url}
\usepackage{listings}
\usepackage{color}
\usepackage{courier}
\usepackage{xspace}
\usepackage{caption}
\usepackage{hyperref}
\usepackage[vlined, ruled, oldcommands]{algorithm2e}
\usepackage{algpseudocode}
\usepackage{listings}
\usepackage{tikz}
\usepackage{csquotes}
\usepackage{booktabs}
\usetikzlibrary{calc}

\newcommand{\cP}{\mathcal{P}}
\newcommand\1{{\mathbf 1}}

\newcommand{\R}{\mathbb{R}}

\newcommand{\Dr}{\textnormal{Dr}}
\newcommand{\TGr}{\textnormal{TGr}}
\newcommand{\cl}{\textnormal{cl}}
\newcommand{\rank}{\textnormal{rank}}
\newcommand\conv{{\rm conv}}
\newcommand{\polymake}{\texttt{poly\-make}\xspace}

\definecolor{appcolor}{rgb}{0,0,0}
\definecolor{commentcolor}{rgb}{.3,.3,.3}
\definecolor{stringcolor}{RGB}{103,5,172}
\definecolor{typecolor}{RGB}{140,75,0}
\definecolor{propcolor}{RGB}{0,100,14}
\definecolor{light_gray}{gray}{0.7}

\lstdefinelanguage{pmshell}
{
  basicstyle=\small\ttfamily,
  keywords=[1]{Hypersurface,ValuatedMatroid,SubdivisionOfPoints},
  morekeywords=[2]{N_ISOLATED, N_FAMILIES, POLYNOMIAL, DEGREE, GENUS, BASES, VALUATION_ON_BASES, WEIGHTS, N_ELEMENTS, FAR_VERTICES, CONES, POINTS, VERTICES, TIGHT_SPAN, MAXIMAL_POLYTOPES, MONOMIALS, COEFFICIENTS, MAXIMAL_CELLS, RAYS, N_RAYS, N_MAXIMAL_CELLS, FACETS, LINEAR_SPAN, CURVE_EDGE_LENGTHS, VISUAL},
  morekeywords=[3]{map, grep, join, secondary_cone},
  stringstyle=\color{stringcolor},
  keywordstyle=[1]\color{typecolor},
  keywordstyle=[2]\color{propcolor},
  keywordstyle=[3]\color{typecolor},
  numbers=none,
  captionpos=b,
  showspaces=false,
  showstringspaces=false,
  morestring=[b]",
  escapechar=&,
  frame=single,
}

\theoremstyle{definition}
\newtheorem{definition}{Definition}[section]
\newtheorem{example}[definition]{Example}

\newtheorem{remark}[definition]{Remark}

\newtheorem*{convention}{Convention}

\theoremstyle{plain}
\newtheorem{theorem}[definition]{Theorem}

\newtheorem{prop}[definition]{Proposition}
\newtheorem{conjecture}[definition]{Conjecture}

\title[Algorithms for Tight Spans and Tropical Linear Spaces]{Algorithms for\\Tight Spans and Tropical Linear Spaces}
\author{Simon Hampe \and Michael Joswig \and Benjamin Schröter}
\address{Technische Universität Berlin\\
Institut für Mathematik, Sekretariat MA 6-2\\
Straße des 17. Juni 136, 10623 Berlin}
\email{\{hampe,joswig,schroeter\}@math.tu-berlin.de}

\thanks{S.~Hampe receives support from DFG (Priority Program 1489, Grant JO 366/3-2).
  M.~Joswig is partially supported by Einstein Foundation Berlin and DFG (Priority Program 1489, Collaborative Research Centers TRR 109 and TRR 195).}

\begin{document}

\begin{abstract}
  We describe a new method for computing tropical linear spaces and more general duals of polyhedral subdivisions.
  It is based on Ganter's algorithm (1984) for finite closure systems.
\end{abstract}

\maketitle

\section{Introduction}
\noindent
Tropical linear spaces are among the most basic objects in tropical geometry~\cite[Chapter 4]{Tropical+Book}.
In polyhedral geometry language they form polyhedral complexes which are dual to regular matroid subdivisions of hypersimplices.
Such subdivisions are characterized by the property that each cell is the convex hull of characteristic vectors of the bases of a matroid.
Here the hypersimplices correspond to the uniform matroids.
Research on matroid subdivisions and related objects goes back to Dress and Wenzel \cite{DressWenzel:1992} and to Kapranov \cite{Kapranov93}.
Speyer instigated a systematic study in the context of tropical geometry \cite{Speyer:2008}, while
suitable algorithms have been developed and implemented by Rinc\'on \cite{Rincon:2013}.

Here we present a new algorithm for computing tropical linear spaces, which is implemented in the software system \polymake \cite{DMV:polymake}.
Moreover, we report on computational experiments.
Our approach has two key ingredients.
First, our method is completely polyhedral --- in contrast with Rinc\'on's algorithm \cite{Rincon:2013}, which primarily rests on exploiting matroid data.
Employing the polyhedral structure has the advantage that this procedure naturally lends itself to interesting generalizations and variations.
In particular, this includes tropical linear spaces corresponding to non-trivially valuated matroids.
Second, our method fundamentally relies on an algorithm of Ganter \cite{Ganter:1984,GanterReuter:1991} for enumerating the closed sets in a finite closure system.
This procedure is a variant of breadth-first-search in the Hasse diagram of the poset of closed sets.
As a consequence the computational costs grow only linearly with the number of edges in the Hasse diagram, i.e., the number of covering pairs among the closed sets.
So this complexity is asymptotically optimal in the size of the output, and this is what makes our algorithm highly competitive in practice.
The challenge is to implement the closure operator and to intertwine it with the search in such a way that it does not impair the output-sensitivity.

Kaibel and Pfetsch employed Ganter's algorithm for enumerating face lattices of convex polytopes \cite{KaibelPfetsch:2002}, and this was later extended to bounded subcomplexes of unbounded polyhedra \cite{HerrmannJoswigPfetsch:2013}.
Here this is generalized further to arbitrary regular subdivisions and their duals.
Such a dual has been called \emph{tight span} in \cite{HerrmannJoswigSpeyer:2012} as it generalizes the tights spans of finite metric spaces studied by Isbell \cite{Isbell:1964} and Dress \cite{Dress:1984}.
The tight span of an arbitrary polytopal complex may be seen as a special case of the dual block complex of a cell complex; e.g., see \cite[\S64]{Munkres:1984}.
From a topological point of view subdivisions of point configurations are cell decompositions of balls, which, in turn, are special cases of manifolds with boundary.
The duality of manifolds with boundary is classically known as Lefschetz duality (e.g., see \cite[\S70]{Munkres:1984}), and this generalizes Poincar\'e duality as well as cone polarity.
With an arbitrary polytopal subdivision, $\Sigma$, we associate a new object, called the \emph{extended tight span} of $\Sigma$, which contains the tight span, but which additionally takes duals of certain boundary cells into account.
In general, the extended tight span is only a partially ordered set.
If, however, $\Sigma$ is regular, then the extended tight span can be equipped with a natural polyhedral structure.
We give an explicit coordinatization.
In this way tropical linear spaces arise as the extended tight spans of matroid subdivisions with respect to those boundary cells which correspond to loop-free matroids.
While a tropical linear space can be given several polyhedral structures, the structure as an extended tight span is the coarsest.
Algorithmically, this has the advantage of being the sparsest, i.e., being the one which takes the least amount of memory.
In this sense, this is the canonical polyhedral structure of a tropical linear space.

This paper is organized as follows.
We start out with recalling basic facts about general closure systems with a special focus on Ganter's algorithm \cite{Ganter:1984}.
Next we introduce the extended tight spans, and this is subsequently specialized to tropical linear spaces.
We compare the performance of Rinc\'on's algorithm \cite{Rincon:2013} with our new method.
To exhibit one application the paper closes with a case study on the $f$-vectors of tropical linear spaces.

\section{Closure Systems, Lower Sets and Matroids}
\noindent
While we are mainly interested in applications to tropical geometry, it turns out that it is useful to start out with some fundamental combinatorics.
This is the natural language for Ganter's procedure, which we list as Algorithm~\ref{algorithm:closed_sets} below.
\begin{definition}
  A \emph{closure operator} on a set $S$ is a function $\cl: \mathcal{P}(S) \to \mathcal{P}(S)$ on the power set of $S$, which fulfills the following axioms for all subsets $A,B\subseteq S$:
 \begin{enumerate}
 \item[(i)] $A \subseteq \cl(A)$ (Extensiveness).
 \item[(ii)] If $A \subseteq B$ then $\cl(A) \subseteq \cl(B)$ (Monotonicity).
 \item[(iii)] $\cl(\cl(A)) = \cl(A)$ (Idempotency).
 \end{enumerate}
 A subset $A$ of $S$ is called \emph{closed}, if $\cl(A) = A$.
 The set of all closed sets of $S$ with respect to some closure operator is called a \emph{closure system}.
\end{definition}
The closed sets of a closure system form a meet-semilattice.
Conversely, each meet-semilattice arises in this way.

Classical examples include the following.
If the set $S$ carries a topology then the function which sends any subset $A$ to the smallest closed set (defined as the complement of an open set) containing $A$ is a closure operator, called the \emph{topological closure}.
If the set $S$ is equipped with a group structure then the function which sends any subset $A$ to the smallest subgroup containing $A$ is a closure operator.
Throughout the following we are particularly interested in the case where the set $S=[n]$ is finite.

The closed sets of a closure system $(S,\cl)$ are partially ordered by inclusion.
The resulting poset is the \emph{closure poset} induced by $(S,\cl)$.
The \emph{Hasse diagram} of $(S,\cl)$ is the directed graph whose nodes are the closed sets and whose arcs correspond to the covering relations of the closure poset.
We assume that all arcs are directed upward, i.e., toward the larger set.
Ganter's Algorithm~\ref{algorithm:closed_sets} computes the Hasse diagram of a finite closure system; see \cite{Ganter:1984,GanterReuter:1991,GanterObiedkov:2016}.
As its key property each closed set is pushed to the queue precisely once, and this entails that the running time is linear in the number of edges of the Hasse diagram, i.e., the algorithm is \emph{output-sensitive}.

\begin{algorithm}
 \caption{Produces the Hasse diagram of a finite closure system.}\label{algorithm:closed_sets}
   \dontprintsemicolon
   \KwIn{A set $S$ and a closure operator $\cl$ on $S$}
   \KwOut{The Hasse diagram of $(S,\cl)$}
   $H$ $\leftarrow$ empty graph\;
   Queue $\leftarrow$ [cl($\emptyset$)]\;
   add node for closed set cl($\emptyset$) to $H$\;
   \While{Queue is not empty}{
   $N$ $\leftarrow$ first element of Queue, remove $N$ from Queue\;
      \ForAll{minimal $N_i :=  \cl(N \cup \{i\})$, where $i \in S \backslash N$}{
         \If{$N_i$ does not occur as a node in $H$ yet}{
            add new node for closed set $N_i$ to $H$\;
            add $N_i$ to Queue\;
         }
      add arc from $N$ to $N_i$ to $H$\;
      }
   }
\Return $H$\;
\end{algorithm}

\begin{example}\label{exmp:polytope}
  Based on Algorithm~\ref{algorithm:closed_sets}, Kaibel and Pfetsch \cite{KaibelPfetsch:2002} proposed a method to compute the face lattice of a convex polytope $P$. This can be done in two different ways:
  A face of a polytope can either be identified by its set of vertices or by the set of facets it is contained in.
  
  In the first case, the set $S$ is the set of vertices and the closure of a set is the smallest face containing this set.
  
  In the second case, the set $S$ is the set of facets. Let $F \subseteq S$. The intersection of the facets in $F$ is a face $Q_F$ of $P$. The closure of $F$ is defined as the set of all facets which contain $Q_F$. Note that, with this approach, Algorithm \ref{algorithm:closed_sets} actually computes the face lattice with inverted relations.
  
  In both cases, the closed sets are exactly the faces of $P$ and the closure operator is given in terms of the vertex--facet incidences.
\end{example}

\begin{example}\label{example:fan}
 Instead of polytopes, one can also compute the face lattice of a polyhedral fan in much the same manner. The crucial problem is to define the closure of a set of rays which is not contained in any cone. The solution to this is to extend the set $S$ to contain not only all rays, but also an additional artificial element, say $\infty$. Now the closure of a set $F \subseteq S$ is either the smallest cone containing it, if it exists, or the full set $S$. In particular, this ensures that the length of a maximal chain in the face lattice of a $k$-dimensional fan is always $k+1$.
\end{example}

The following class of closure systems is ubiquitous in combinatorics and tropical geometry.
The monographs by White~\cite{White:1986} and Oxley~\cite{Oxley:2011} provide introductions to the subject.
\begin{definition}\label{def:matroid}
  Let $S$ be a finite set equipped with a closure operator $\cl: \mathcal{P}(S) \to \mathcal{P}(S)$.
  The pair $(S,\cl)$ is a \emph{matroid} if the following holds in addition to the closure axioms:
  \begin{enumerate}
  \item[(iv)] If $A\subseteq S$ and $x\in S$, and $y\in\cl(A\cup\{x\})\setminus\cl(A)$, then $x\in\cl(A\cup\{y\})$ (MacLane--Steinitz Exchange).
  \end{enumerate}
\end{definition}
This is one of many ways to define a matroid; see \cite[Lemma~1.4.3]{Oxley:2011} for explicit cryptomorphisms.
The closed sets of a matroid are called \emph{flats}.
\begin{remark}
  For matroids it is not necessary to check for the minimality of the closed sets $N_i$ in Algorithm~\ref{algorithm:closed_sets}.
  In view of Axiom~(iv) this is always satisfied. This application of the algorithm also demonstrates that, while the empty set is typically closed, this does not always need to be the case. In fact, the closure of the empty set in a matroid is the set of all \emph{loops}.
\end{remark}

For our applications it will be relevant to consider special closure systems which are derived from other closure systems in the following way.
A \emph{lower set} $\Lambda$ of the closure system $(S,\cl)$ is a subset of the closed sets such that for all pairs of closed sets with $A\subseteq B$ we have that $B\in \Lambda$ implies $A\in \Lambda$.
\begin{prop}\label{prop:lower}
  Let $(S,\cl)$ be a closure system with lower set $\Lambda$.
  Then the function $\cl_\Lambda$ which is defined by
  \begin{equation}\label{eq:cl_Lambda}
    \cl_\Lambda(A) \ = \begin{cases} \cl(A) & \text{if $\cl(A)\in\Lambda$} \enspace,\\ S & \text{otherwise} \end{cases}
  \end{equation}
  is a closure operator on $S$.
\end{prop}
\begin{proof}
  Extensiveness and idempotency are obvious.
  We need to show that monotonicity holds.
  To this end consider two closed sets $A\subseteq B\subseteq S$.
  Suppose first that $A$ lies in the lower set $\Lambda$.
  Then $\cl_\Lambda(A)=\cl(A)\subseteq\cl(B)\subseteq\cl_\Lambda(B)$.
  If, however, $A\not\in\Lambda$, then $B\not\in\Lambda$ as $\Lambda$ is a lower set.
  In this case we have $\cl_\Lambda(A)=S=\cl_\Lambda(B)$.
\end{proof}

\begin{example}
  An unbounded convex polyhedron is \emph{pointed} if it does not contain any affine line.
  In that case the polyhedron is projectively equivalent to a convex polytope, with a marked face, the \emph{face at infinity}; see, e.g., \cite[Theorem~3.36]{Polyhedral+Methods}.
  So we arrive at the situation where we have a convex polytope $P$ with a marked face $F$.
  Now the set of faces of $P$ which do not intersect $F$ trivially forms a lower set $\Lambda$ in the closure system of faces of $P$.
  In this way combining Example~\ref{exmp:polytope} with Proposition~\ref{prop:lower} and Algorithm~\ref{algorithm:closed_sets} gives a method to enumerate the bounded faces of an unbounded polyhedron.
  Ignoring the entire set $S$, which is closed with respect to $\cl_\Lambda$ but not bounded, recovers the main result from \cite{HerrmannJoswigPfetsch:2013}.
\end{example}

\begin{example}
  For a $d$-polytope $P$ and $k\leq d$ the faces of dimension at most $k$ form a lower set.
  This is the \emph{$k$-skeleton} of $P$.
\end{example}

The closure operators from all examples in this section are implemented in \polymake \cite{DMV:polymake}.

\section{Extended Tight Spans}\label{sec:tight_spans}
\noindent
It is the goal of this section to describe duals of polytopal complexes in terms of closure systems.
Via Algorithm~\ref{algorithm:closed_sets} this gives means to deal with them effectively.
For details on polyhedral subdivisions we refer to the monograph \cite{Triangulations}.

Let $\cP\subset\R^d$ be a finite point configuration, and let $\Sigma$ be a polytopal subdivision of $\cP$.
That is, $\Sigma$ is a polytopal complex whose vertices lie in $\cP$ and which covers the convex hull $P = \conv  \cP$.
We call the elements of $\Sigma$ \emph{cells}; the set of maximal cells is denoted by $\Sigma^{\max}$ and the maximal boundary facets (meaning the maximal cells of $\Sigma$ contained in the facets of $P$) by $\Delta_\Sigma$.
Now we obtain a closure operator on the set $S_\Sigma := \Sigma^{\max} \cup \Delta_\Sigma$ by letting
\begin{equation}
 \cl^\Sigma(F) \ := \   
  \begin{cases}
    \emptyset & \textnormal{if } F=\emptyset \enspace, \\
    \left\{ g \in S_\Sigma\;\middle\vert\; \bigcap_{\sigma \in F} \sigma \subseteq g\right\} & \textnormal{otherwise}\enspace.
  \end{cases}
\end{equation}
for any $F\in S_\Sigma$.
Note that $\cl^\Sigma$ is basically the same as the dual operator in Example \ref{exmp:polytope}.
In fact, the closed, non-empty sets in $S_\Sigma$ correspond to the cells of $\Sigma$, while the poset relation is the inverse containment relation.

Now let $\Gamma$ be a collection of boundary faces of $\Sigma$. This defines a lower set $\Lambda_\Gamma$ for the closure system $(S_\Sigma, \cl^\Sigma)$, which consists of all sets $F$, such that $\bigcap_{\sigma \in F} \sigma \nsubseteq \tau$ for any $\tau \in \Gamma$. We will denote the corresponding closure operator by $\cl^\Sigma_\Gamma := \cl^\Sigma_{\Lambda_\Gamma}$ and we call the resulting closure system $(S_\Sigma, \cl^\Sigma_\Gamma)$ the \emph{extended tight span} of $\Sigma$ with respect to $\Gamma$. 

If $\Gamma = \Delta_\Sigma$, the closed sets are all cells of $\Sigma$ which are not contained in the boundary.
This is exactly the \emph{tight span} of a polytopal subdivision defined in \cite{HerrmannJensenJoswigSturmfels:2009}, which is dual to the interior cells.
Note that this can also be obtained as the closure system of $(\Sigma^{\max},\cl^\Sigma)$.

\begin{example}
   Let $\cP$ be $\{-1,0,1\}$ with the convex hull $P = [-1,1]$, and let $\Sigma$ be its subdivision into intervals $[-1,0]$ and $[0,1]$. The subdivision and its corresponding extended tight spans for $\Gamma = \emptyset$ and $\Gamma = \Delta(\Sigma)$ can be seen in Figure \ref{figure:interval}. This example also demonstrates that we need to declare the closure of the empty set to be itself to ensure monotonicity.
 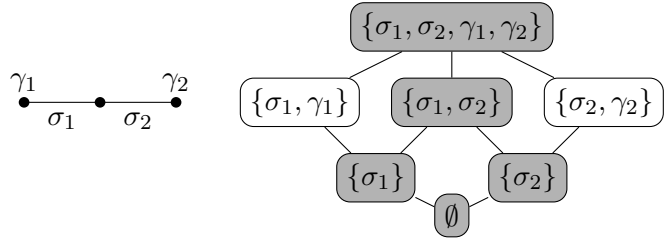
\begin{figure}
  \centering
  \begin{tikzpicture}
   \newcommand*{\stepup}{1}
   \newcommand*{\stepright}{2}
   \matrix[column sep=15pt]{
    \draw (-1,0) node[above] {$\gamma_1$} -- (1,0) node[above]{$\gamma_2$};
    \fill[black] (-1,0) circle (2pt);
    \fill[black] (0,0) circle (2pt);
    \fill[black] (1,0) circle (2pt);
    \node at (-.5,0)[below] {$\sigma_1$};
    \node at (.5,0)[below] {$\sigma_2$};
   &
   \node[rounded corners =5pt,draw,fill=light_gray] (zero) at (0,-1.5*\stepup) { $\emptyset$};
   \node[rounded corners =5pt,draw,fill=light_gray] (s1) at (-0.5*\stepright,-1*\stepup) { $\{\sigma_1\}$};
   \node[rounded corners =5pt,draw,fill=light_gray] (s2) at (0.5*\stepright,-1*\stepup) { $\{\sigma_2\}$};
   \node[rounded corners =5pt,draw] (s1g1) at (-1*\stepright,0) { $\{\sigma_1,\gamma_1\}$};
   \node[rounded corners =5pt,draw,fill=light_gray] (s1s2) at (0,0) { $\{\sigma_1,\sigma_2\}$};
   \node[rounded corners =5pt,draw] (s2g2) at (1*\stepright,0) { $\{\sigma_2,\gamma_2\}$};
   \node[rounded corners =5pt,draw,fill=light_gray] (full) at (0,1*\stepup) { $\{\sigma_1,\sigma_2,\gamma_1,\gamma_2\}$};
   \foreach \x in {(s1g1),(s1s2),(s2g2)}
      {\draw \x -- (full);}
   \draw (s1) -- (s1g1);
   \draw (s1) -- (s1s2);
   \draw (s2) -- (s1s2);
   \draw (s2) -- (s2g2);
   \draw (zero) -- (s1);
   \draw (zero) -- (s2);
  \\
  };
  \end{tikzpicture}
 \caption{A regular subdivision and its extended tight span for $\Gamma = \emptyset$ and $\Gamma = \Delta(\Sigma) = \{\gamma_1,\gamma_2\}$, respectively. The latter is marked in gray.}\label{figure:interval}
 \end{figure}

\end{example}

If the subdivision is regular, i.e., induced by a height function $h: \cP \to \R$, we can actually coordinatize the extended tight span. Any regular subdivision with fixed height function is dual to a \emph{dual complex} $N_\Sigma$, which is a complete polyhedral complex in $\R^d$. More precisely, for every point $x \in \R^d$ there is a cell of $\Sigma$, consisting of all points $p \in \cP$ which minimize $h(p) - p \cdot x$. Points which induce the same cell form an open polyhedral cell and the topological closures of these cells form the dual complex. It is a well-known fact that there is a bijective, inclusion-reversing relation between the cells of $\Sigma$ and the cells of $N_\Sigma$.

In particular, each maximal cell of $\Sigma^{\max}$ is dual to a vertex and every boundary facet in $\Delta(\Sigma)$ is dual to a ray of $N_\Sigma$. Hence, every closed set $F$ of $(S_\Sigma,\cl^\Sigma_\Gamma)$ corresponds to a polyhedral cell $\rho_F$ and together these cells form a subcomplex of $N_\Sigma$. More precisely, we denote by
\begin{equation}
T_{\Sigma,\Gamma} := \left\{ \rho_F \;\middle\vert\;F \subsetneq S_\Sigma \textnormal{ closed w.r.t.\ } \cl^\Sigma_\Gamma\right\}
\end{equation}
the \emph{coordinatized extended tight span} of $\Sigma$ with respect to $\Gamma$. Its face lattice is by definition the poset of closed sets of $(S_\Sigma,\cl^\Sigma_\Gamma)$.

\section{Tropical Linear Spaces}\label{sec:tropli}
\noindent
In this section we will finally investigate the objects that we are most  interested in: valuated matroids and tropical linear spaces.
We prefer to see the latter as special cases of extended tight spans.
Valuated matroids were first studied by Dress and Wenzel \cite{DressWenzel:1992}; see \cite[Chapter 4]{Tropical+Book} for their role in tropical geometry.

Let us introduce some notation.
For a subset $B$ of $[n]$ of size $r$, let $e_B := \sum_{i\in B} e_i$.
For a collection $M \subseteq \binom{[n]}{r}$ of such subsets we let
\begin{equation}
P_M \ := \ \conv\{e_B \mid B \in M\}
\end{equation}
be the subpolytope of the hypersimplex $\Delta(r,n)$ which is spanned by those vertices which correspond to elements in $M$.
In this language matroids were characterized by Gel’fand, Goresky, MacPherson and Serganova \cite{GelfandEtAl:1987} as follows.
\begin{prop}\label{prop:matroid}
  The set $M$ comprises the bases of a matroid if and only if the vertex--edge graph of the polytope $P_M$ is a subgraph of the vertex--edge graph of $\Delta(r,n)$ or, equivalently, if every edge of $P_M$ is parallel to $e_i - e_j$ for some $i$ and~$j$.
\end{prop}

Throughout the following, let $M$ be (the set of bases of) a matroid on $n$ elements.
In that case $P_M$ is the \emph{matroid polytope} of $M$.
The matroid $M$ is said to be \emph{loop-free} if $\bigcup_{B \in M} B = [n]$.
The \emph{rank} of $M$ is $r$, the size of any basis.
If $P_M$ is the full hypersimplex, then $M = U_{r,n}$ is a \emph{uniform matroid}.
The above description fits well with our geometric approach.
Any function $v: M \to \R$ gives rise to a regular subdivision on $P_M$, which we denote by $\Sigma_{M,v}$.
The pair $(M,v)$ is a \emph{valuated matroid} if every cell of $\Sigma_{M,v}$ is again a matroid polytope.
Then $\Sigma_{M,v}$ is called a \emph{matroid subdivision}.

\begin{example}
  The set $M$ of subsets of $\{1,2,3,4\}$ with exactly two elements has cardinality six.
  Their characteristic vectors are the vertices of a regular octahedron embedded in $4$-space.
  If we let $v$ be the map which sends five vertices to $0$ and the sixth one to $1$, then $(M,v)$ is a valuated matroid.
\end{example}

We will define tropical linear spaces as duals of valuated matroids.
To this end let $(M,v)$ be a valuated matroid of rank $r$ on $n$ elements.
For a vector $x \in \R^n$, we define the set
\begin{equation}
 M_x \ := \ \left\{ B \in M \;\middle\vert\; v(B) - e_B \cdot x \textnormal{ is minimal}\right\}\enspace. 
\end{equation}

From the definition of the dual complex in Section \ref{sec:tight_spans} we see that the elements of $M_x$ correspond to a cell of $\Sigma_{M,v}$ and thus define a matroid.
Note that for any $\lambda \in \R$ we clearly have $M_x = M_{x + \lambda \1}$.

\begin{definition}
  The \emph{tropical linear space} associated with the valuated matroid $(M,v)$ is the set
  \begin{equation}
   B(M,v) \ := \ \left\{ x \in \R^n\;\mid\; M_x \textnormal{ is loop-free}\right\}/\R\1\enspace.
  \end{equation}

\end{definition}

Note that our definition of a valuated matroid, as well as that of a tropical linear space are with respect to minimum as tropical addition. 
The following is our main result.
While it is easy to prove, it is relevant since it entails a new effective procedure for enumerating the cells of a tropical linear space via Algorithm~\ref{algorithm:closed_sets}.
\begin{theorem}
  Let $\Gamma$ be the set of boundary faces of $\Sigma := \Sigma_{M,v}$ which correspond to matroids with loops.
  Then
  \begin{equation}
   B(M,v) \ = \ T_{\Sigma,\Gamma}/\R\1\enspace,
  \end{equation}
where $T_{\Sigma,\Gamma}$ is the coordinatized extended tight span defined in Section \ref{sec:tight_spans}.
\end{theorem}
\begin{proof}
 Let $N_\Sigma$ denote the dual complex of $\Sigma$. From our definition it is immediately clear that $B(M,v)$ is a subcomplex of $N_\Sigma / \R\1$. 
 It consists of all cells whose dual cell in $\Sigma$ is the polytope of a loop-free matroid. Since any cell in $\Sigma$ corresponds to a loop-free matroid, if and only if it is not contained in a boundary facet of a matroid with loops, the claim follows.
\end{proof}
We call the resulting polyhedral structure of $B(M,v)$ \emph{canonical}.
\begin{remark}
 Note that one can naturally replace $\Gamma$ by the subset of \emph{maximal} boundary faces corresponding to matroids with loops. These faces are defined by the equations $x_i = 0$ for $i \in [n]$.
\end{remark}
\begin{example}
  If the valuation is constant then the matroid subdivision is trivial.
  It follows that the dual complex coincides with the normal fan of the matroid polytope $P_M$.
   In this case $B(M,v)$ is the \emph{Bergman fan} of $M$, in its coarsest possible subdivision; see \cite{HampeThesis2014} for a proof.
\end{example}
The polyhedral complex $B(M,v)$ reflects quite a lot of the combinatorics of the matroid $M$.
For instance, the rank of $M$ equals $\dim(B(M,v)) + 1$.
If $L$ is the lineality space of $B(M,v)$, then the number of connected components of $M$ is $\dim(L) +1$; see \cite{FeichtnerSturmfels:2005}.

\subsection{Performance comparison}\label{subsec:performance}
As mentioned in the introduction, there is an algorithm by Rinc\'{o}n \cite{Rincon:2013} for computing Bergman fans, i.e., tropical linear spaces with trivial valuation.
An extension which can also deal with trivially valuated arbitrary matroids which may not be realizable has been implemented in \polymake's bundled extension \texttt{a-tint} \cite{hatint}.
It is this implementation we refer to in the following discussion.
The original software \texttt{TropLi} by Rinc\'on only takes realizable matroids as input.

\begin{table}[th]
   \centering
   \caption{Comparing running times for computing Bergman fans.}
   \label{table:bergman}
   \renewcommand{\arraystretch}{0.9}
   \begin{tabular*}{0.95\linewidth}{@{\extracolsep{\fill}}l@{\hspace{0.2cm}}r@{\hspace{0.7cm}}rr@{\hspace{0.7cm}}rr@{}}
   \toprule
      $(n,r)$ & $\#$ matroids & Rinc\'{o}n & Hasse & CH & ETS\\
   \midrule
   (6,2) & 23 & 0.0 & 0.2 & 0.8 & 0.0 \\
      (6,3) & 38 & 0.0 & 0.4 & 1.6 & 0.0 \\[3pt]
   (7,2) & 37 & 0.0 & 0.3 & 1.6 & 0.0 \\
      (7,3) & 108 & 0.0 & 1.5 & 5.8 & 0.2 \\[3pt]
   (8,2) & 58 & 0.0 & 0.4 & 1.9 & 0.0 \\
   (8,3) & 325 & 0.3 & 6.0 & 21.4 & 0.8 \\
   (8,4) & 940 & 1.8 & 48.7 & 86.5 & 9.2 \\
   \bottomrule
 \end{tabular*}
\end{table}

Rinc\'{o}n's and our algorithm are very difficult to compare for two reasons.
First of all, a matroid can be encoded in numerous ways. 
For instance, in terms of closures, as in Definition~\ref{def:matroid}, or in terms of bases, as in Proposition~\ref{prop:matroid}.
Many further variants exist, and the conversion between these representations is often a non-trivial computational task. 
Below we will assume that all matroids are given in terms of their bases. 
The second problem is that the two algorithms essentially compute very different things. 
Our algorithm computes the full face lattice of the canonical polyhedral structure of a tropical linear space. 
On the other hand, Rinc\'{o}n's algorithm only computes the rays and the maximal cones of the Bergman fan, albeit in a finer subdivision. 
In this setup it is therefore to be expected that our approach is significantly slower. 
In particular, to identify the boundary cells (including the loopfree ones), we need to apply a convex hull algorithm to the matroid polytope before we can make use of our algorithm.
Still, the discussion has merit when separating the timings for the different steps; see Table \ref{table:bergman}.
We compute Bergman fans of all (isomorphism classes of) matroids of a given rank $r$ on a given ground set $[n]$ as provided at \url{http://www-imai.is.s.u-tokyo.ac.jp/~ymatsu/matroid}); see also \cite{Matsumotoetal:2012}.
Each matroid is given only in terms of its bases. 
We first apply Rinc\'{o}n's algorithm and then compute the Hasse diagram of the face lattice of the fan as described in Example~\ref{example:fan}. 
For our approach we split the computations into two steps: 
First we compute the convex hull of the matroid polytope, displayed under \enquote{CH} and then measure the running time of our closure algorithm \enquote{ETS} (extended tight span) separately. 
Times were measured on an AMD Phenom II X6 1090T with 3.6 GHz using a single thread and \polymake version 3.1.
We employed the double description method implemented in the \texttt{Parma Polyhedral Library} (via \polymake's interface) for computing the convex hulls \cite{PPL}.

The results show that almost all of the time in our algorithm is spent computing the facets of the matroid polytope. 
On the other hand, if one aims at obtaining the same amount of information, i.e., the full face lattice, for Rinc\'{o}n's algorithm, this increases the computation time dramatically. This demonstrates that the finer subdivision produced by this algorithm is significantly worse in terms of complexity than the canonical subdivision.

We also like to point out that for non-trivial valuations our algorithm is, to the best of our knowledge, currently the only feasible method for computing tropical linear spaces.

\section{A case study on \texorpdfstring{$f$}{f}-vectors of tropical linear spaces}
\noindent
Throughout the rest of this paper we will restrict ourselves to valuations of uniform matroids.
Equivalently, we study matroid subdivisions of hypersimplices (and their lifting functions).
Speyer was the first to conduct a thorough study of the combinatorics of tropical linear spaces \cite{Speyer:2008}.
He conjectured the following.
\begin{conjecture}[Speyer's $f$-vector conjecture]\label{con:speyer}
 Let $n \geq 1$ and $1 \leq r \leq n$. Let $v$ be any valuation on $U_{r,n}$. Then the number of $(i-1)$-dimensional bounded faces of $B(U_{r,n},v)$ is at most $\binom{n-2i}{r-i} \binom{n-i-1}{i-1}$.
\end{conjecture} 

To study this problem, one is naturally interested in some form of moduli space of all possible valuations on $U := U_{r,n}$.
This role is played by the \emph{Dressian} $\Dr(r,n)$; see \cite{HerrmannJensenJoswigSturmfels:2009,HerrmannJoswigSpeyer:2012}.
It is a subfan of the secondary fan of $P_U = \Delta(r,n)$, consisting of all cones which correspond to matroidal subdivisions.
As a set it contains the \emph{tropical Grassmannian} $\TGr_p(r,n)$ introduced by Speyer and Sturmfels \cite{SpeyerSturmfels:2004} for any characteristic $p$.
This is the tropicalization of the ordinary complex Grassmannian, and it consists of all cones of the secondary fan which correspond to \emph{realizable} valuations on $U$, i.e., those which can be realized as valuated vector matroids in characteristic zero \cite[Chapter 4]{Tropical+Book}.
However, this inclusion is generally strict.
In fact, the Dressian is not even pure in general.

\begin{remark}
  For $r = 2$, the Dressian $\Dr(r,n)$ is equal to the tropical Grassmannian.
  Combinatorially, this is the \emph{space of phylogenetic trees}; e.g., see \cite[\S1.3]{Kapranov93} and \cite[\S4.3]{Tropical+Book}. 
  For $r = 3$ and $3 \leq n \leq 6$, the equality $\Dr(r,n) = \TGr_p (r,n)$ still holds on the level of sets for each $p\geq 0$. This is trivial for $n = 3,4$, as there are no non-trivial subdivisions of $P_U$ in that case. For $n = 5$ it follows from duality and the statement for $\Dr(2,5)$. The Dressian $\Dr(3,6)$ was computed in \cite{SpeyerSturmfels:2004}. Note that, while the Dressian and the Grassmannian may agree as sets, they can have different polyhedral structures. Understanding the precise relation between these structures is still an open problem for general parameters.
  The cases $(3,7)$ and $(3,8)$ are the first where the Dressian differs from the Grassmannian. 
  The Dressian $\Dr(3,7)$ was computed in \cite{HerrmannJensenJoswigSturmfels:2009}. 
  In particular, the possible combinatorial types of the corresponding tropical planes (and thus, their possible $f$-vectors) were listed. 
  The polyhedral structure of $\Dr(3,8)$ was computed in \cite{HerrmannJoswigSpeyer:2012}.
\end{remark}

\subsection{The Dressian \texorpdfstring{$\Dr(3,8)$}{Dr(3,8)}}

We wish to compute $f$-vectors of uniform tropical planes in $\R^8/\R\1$, i.e., tropical linear spaces corresponding to valuations on $U_{3,8}$. To this end, we make use of the data obtained in \cite{HerrmannJoswigSpeyer:2012}, which is available at \url{http://svenherrmann.net/DR38/dr38.html}. There is a natural $S_8$-symmetry on the Dressian and the web page provides representatives for each cone orbit. 

We computed tropical linear spaces for each cone by choosing an interior point as valuation. 
%
%
For the sake of legibility, we only include results for the maximal cones of the Dressian. There are 14 maximal cones of dimension nine and 4734 maximal cones of dimension eight. The full data can be obtained at \url{http://page.math.tu-berlin.de/~hampe/dressian38.php}.

\begin{convention}
  The $f$-vector of a tropical linear space $L$ is the $f$-vector of its canonical polyhedral structure.
  The \emph{bounded $f$-vector} of $L$ is the $f$-vector of the bounded part of this structure. All counts are given modulo the $S_8$-symmetry on the Dressian.
\end{convention}

There is only one bounded $f$-vector $(n-2, n-3)$ for a tropical linear space that corresponds to a maximal cone in the Dressian $\Dr(2,n)$, since this linear space has the combinatorics of a binary tree with $n$ labeled leaves.
The generic tropical linear spaces in the Dressian $\Dr(3,6)$ have a bounded $f$-vector which is either $(5, 4, 0)$ or $(6, 6, 1)$; see \cite{HerrmannJensenJoswigSturmfels:2009}.
In the case of $(3,7)$ the (generic) bounded $f$-vectors read $(7,6,0)$, $(9, 10, 2)$ and $(10, 12, 3)$.

\begin{theorem}
 Every generic tropical plane in $\R^8/\R\1$ has one of four possible $f$-vectors:
 \begin{itemize}
  \item[$\rhd$] If it corresponds to a nine-dimensional cone in the Dressian, its $f$-vector is $(13, 55, 63)$ and its bounded $f$-vector is $(13,15,3)$. There are nine different combinatorial types of such planes; see Figure \ref{figure:combinatorial_types}.
  \item[$\rhd$] If it corresponds to one of the 4734 eight-dimensional maximal cones in the Dressian, there are three possibilities:
  \begin{itemize}
   \item[$\circ$] There are 51 planes with $f$-vector $(13, 56, 64)$ and bounded $f$-vector $(13, 16, 4)$.
   \item[$\circ$] There are 1079 planes with $f$-vector $(14, 58, 65)$ and bounded $f$-vector $(14, 18, 5)$.
   \item[$\circ$] There are 3604 planes with $f$-vector $(15, 60, 66)$ and bounded $f$-vector $(15, 20, 6)$.
  \end{itemize}
  There are 3013 different combinatorial types of such planes.
 \end{itemize}
\end{theorem}

The maximal bounded $f$-vector $(15,20,6)$ agrees with the upper bound in Conjecture~\ref{con:speyer}.

\begin{figure}[ht]
 \centering
  \begin{tikzpicture}[x={(.8cm,-.2cm)}, y={(0.5cm,0.3cm)},z={(0cm,.7cm)}]
  \coordinate (B) at (0,0,0);
  \coordinate (B1) at (-.5,-.5,-1);
  \coordinate (B2) at (-.5,.5,-1);
  \coordinate (B3) at (.5,-.5,-1);
  \coordinate (B4) at (.5,.5,-1);
  \coordinate (T) at (0,0,2);
  \coordinate (LT) at (-1,0,2);
  \coordinate (LTS) at (-1.5,0,2.5);
  \coordinate (LBS) at (-1.5,0,-.5);
  \coordinate (LB) at (-1,0,0);
  \coordinate (MT) at (1,0,2);
  \coordinate (MB) at (1,0,0);
  \coordinate (MTS) at (1.5,0,2.5);
  \coordinate (MBS) at (1.5,0,-.5);
  \coordinate (MTS) at (2,0,2.5);
  \coordinate (RT) at (0,1,2);
  \coordinate (RB) at (0,1,0);
  \coordinate (RTS) at (0,1.5,2.5);
  \coordinate (RBS) at (0,1.5,-.5);
  \coordinate (PS) at (1.5,0,1);
  
  \newcommand\typeline[2]{
  \draw (#1) -- (#2);
  \fill[black] (#2) circle (2pt);
  }
  \newcommand\bottomline{
  \typeline{0,0,0}{-.5,-.5,-1};
  \typeline{0,0,0}{-.5,.5,-1};
  \typeline{0,0,0}{.5,-.5,-1};
  \typeline{0,0,0}{.5,.5,-1};
  \fill[white,draw=black] (0,0,0) circle (2pt);
  }
  \newcommand\righttriangle{
    \draw[fill=gray, fill opacity = 0.4] (B) -- (T) -- (RT) -- (B);
    \draw (RT) -- (RTS);
    \fill[black] (B) circle (2pt);
    \fill[black] (T) circle (2pt);
    \fill[black] (RT) circle (2pt);
    \fill[black] (RTS) circle (2pt);
  }
  \newcommand\lefttriangle{
    \draw[fill=gray, fill opacity = 0.4] (B) -- (T) -- (LT) -- (B);
    \draw (LT) -- (LTS);
    \fill[black] (B) circle (2pt);
    \fill[black] (T) circle (2pt);
    \fill[black] (LT) circle (2pt);
    \fill[black] (LTS) circle (2pt);
  }
  \newcommand\middletriangle{
    \draw[fill=gray, fill opacity = 0.4] (B) -- (T) -- (MT) -- (B);
    \draw (MT) -- (MTS);
    \fill[black] (B) circle (2pt);
    \fill[black] (T) circle (2pt);
    \fill[black] (MT) circle (2pt);
    \fill[black] (MTS) circle (2pt);
  }
  \newcommand\lowsquare{
    \draw[fill=gray, fill opacity = 0.4] (B) -- (T) -- (MT) -- (MB) -- (B);
    \draw (MB) -- (MBS);
    \fill[black] (B) circle (2pt);
    \fill[black] (T) circle (2pt);
    \fill[black] (MT) circle (2pt);
    \fill[black] (MB) circle (2pt);
    \fill[black] (MBS) circle (2pt);
  }
  \newcommand\highsquare{
    \draw[fill=gray, fill opacity = 0.4] (B) -- (T) -- (MT) -- (MB) -- (B);
    \draw (MT) -- (MTS);
    \fill[black] (B) circle (2pt);
    \fill[black] (T) circle (2pt);
    \fill[black] (MT) circle (2pt);
    \fill[black] (MB) circle (2pt);
    \fill[black] (MTS) circle (2pt);
  }
  \newcommand\plainrightsquare{
    \draw[fill=gray, fill opacity = 0.4] (B) -- (T) -- (RT) -- (RB) -- (B);
     \fill[black] (B) circle (2pt);
    \fill[black] (T) circle (2pt);
    \fill[black] (RT) circle (2pt);
    \fill[black] (RB) circle (2pt);
  }
  \newcommand\plainleftsquare{
    \draw[fill=gray, fill opacity = 0.4] (B) -- (T) -- (LT) -- (LB) -- (B);
     \fill[black] (B) circle (2pt);
    \fill[black] (T) circle (2pt);
    \fill[black] (LT) circle (2pt);
    \fill[black] (LB) circle (2pt);
  }
  \newcommand\pentagon{
    \draw[fill=gray, fill opacity = 0.4] (B) -- (T) -- (MT) -- (PS) -- (MB) -- (B);
    \fill[black] (B) circle (2pt);
    \fill[black] (T) circle (2pt);
    \fill[black] (MT) circle (2pt);
    \fill[black] (PS) circle (2pt);
    \fill[black] (MB) circle (2pt);
  }

  \matrix[column sep = 8pt, row sep = 8pt] {
  \lowsquare
  \lefttriangle
  \righttriangle
  \bottomline &
  \highsquare
  \lefttriangle
  \righttriangle
  \bottomline & 
  \lowsquare
  \plainrightsquare
  \lefttriangle 
  \bottomline
  \\
  \node (0,0,0) {S$_*$TT}; &
  \node (0,0,0) {S$^*$TT}; &
  \node (0,0,0) {S$_*$ST}; \\
  \highsquare
  \plainleftsquare
  \righttriangle
  \bottomline & 
  \highsquare
  \plainrightsquare
  \plainleftsquare
  \bottomline &
  \lowsquare
  \plainleftsquare
  \plainrightsquare 
  \bottomline \\
  \node (0,0,0) {S$^*$ST}; &
  \node (0,0,0) {S$^*$SS}; &
  \node (0,0,0) {S$_*$SS}; \\
  \lefttriangle
  \righttriangle
  \pentagon
  \bottomline &
  \pentagon
  \plainrightsquare
  \lefttriangle
  \bottomline & 
  \plainleftsquare
  \plainrightsquare
  \pentagon
  \bottomline \\
  \node (0,0,0) {PTT}; &
  \node (0,0,0) {PST}; &
  \node (0,0,0) {PSS};
  \\
  };
 \end{tikzpicture}
  \caption{The various combinatorial types of bounded parts of tropical linear spaces corresponding to nine-dimensional cones in the Dressian.
    Note that all of them share the same $f$-vector $(13,15,3)$.
    The naming convention is P = pentagon, S  = square, T = triangle.
    The star $*$ indicates where the square has an additional edge attached.}\label{figure:combinatorial_types}
\end{figure}
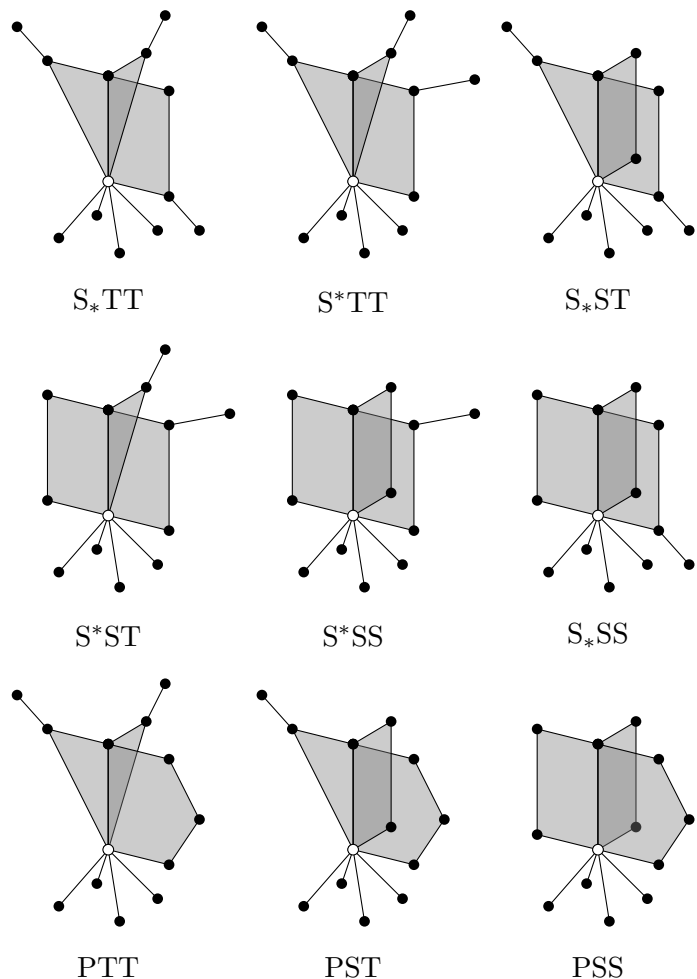

\begin{remark}
   Each of the nine different combinatorial types that correspond to a nine-dimensional cone contain a vertex (marked in white in Figure \ref{figure:combinatorial_types}), which in turn corresponds to the matroid polytope of a parallel extension of the \emph{Fano matroid}. This is a certificate that these tropical linear spaces are not realizable over any field of characteristic greater than two; see \cite[Chapter 6 and Appendix]{Oxley:2011}.
   Figure~\ref{fig:fano_ex} illustrates the connected extensions of the Fano matroid; these are those that are loop-free.
\end{remark}

Further computer experiments reveal the following details.
\begin{prop}
   Let $p$ be $0$, $3$, $5$ or $7$.
   Then the intersection of the relative interior of a cone $C$ in the Dressian $\Dr(3,8)$ with the tropical Grassmannian $\TGr_p(3,8)$ is trivial if and only if a subdivision which is induced by a lifting in the relative interior of $C$ contains the polytope of a Fano matroid extension as a cell.
\end{prop}

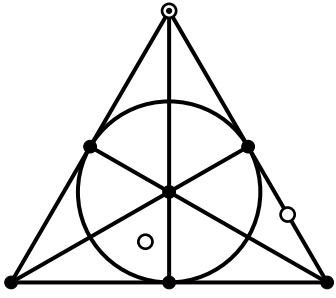
\begin{figure}[ht]
 \centering
   \begin{tikzpicture}[
  edge/.style={line width=1.5pt, line cap=round, black},
  vertex/.style={draw=black, fill=black, line width=0pt},
  new_vertex/.style={draw=black, fill=white, line width=1pt},
  scale=1.2
  ]

  \coordinate (6) at (0,0); 
  \coordinate (0) at (0,2);
  \coordinate (1) at (1.732,-1);
  \coordinate (3) at (-1.732,-1);
  
  \coordinate (2) at ($(0)!0.5!(1)$);  
  \coordinate (4) at ($(0)!0.5!(3)$);
  \coordinate (5) at ($(1)!0.5!(3)$);
  
  \coordinate (e0) at (-1.7,1.7);
  \coordinate (e1) at (0);
  \coordinate (e2) at ($(1)!0.5!(2)$);
  \coordinate (e3) at ($(4)!0.7!(5)$);
    
  
  \draw[draw=black, line width=1.5pt] (6) circle (1);
  \draw[edge] (0) -- (1);
  \draw[edge] (0) -- (3);
  \draw[edge] (0) -- (5);
  \draw[edge] (1) -- (3);
  \draw[edge] (1) -- (4);
  \draw[edge] (2) -- (3);
  
  \draw[new_vertex] (e1) circle (2.2pt);
  \draw[new_vertex] (e2) circle (2.2pt);
  \draw[new_vertex] (e3) circle (2.2pt);
 
  \draw[vertex] (0) circle (0.8pt);
  \draw[vertex] (1) circle (2pt);
  \draw[vertex] (2) circle (2pt);
  \draw[vertex] (3) circle (2pt);
  \draw[vertex] (4) circle (2pt);
  \draw[vertex] (5) circle (2pt);
  \draw[vertex] (6) circle (2pt); 
 
\end{tikzpicture}
  \caption{The three loop-free extensions of the Fano matroid.\label{fig:fano_ex}}
\end{figure}

\section{Outlook}

\subsection{Higher Dressians}

We have given an algorithm which computes tropical linear spaces for arbitrary valuations in reasonable time; computing all tropical linear spaces for $\Dr(3,8)$ above only took a few hours on a standard personal computer. 
This indicates that it is feasible to apply this algorithm more ambitiously, e.g., to Dressians with larger parameters.
However, in these cases not much data is currently available.
Computing higher Dressians is a challenging task in itself.

The next step would be to look at $\Dr(4,8)$.
While computing the full Dressian is, at the moment, beyond our means, we can consider the following construction by Speyer \cite{Speyer:2008}. Let $M$ be a matroid of rank $r$ on $n$ elements. We define an associated valuation on $U_{r,n}$ by
\begin{equation}
 v_M(B) := r - \rank_M(B)\enspace,
\end{equation}
where $B \in \binom{[n]}{r}$ is a basis of the uniform matroid and 
\begin{equation}
 \rank_M(B) = \max_{B' \in M} \{ \lvert B \cap B'\rvert\}
\end{equation}
is the rank of $B$ in $M$. Speyer showed that the corank indeed defines a valuation and that the matroid polytope $P_M$ appears as a cell in the induced regular subdivision.

There are 940 isomorphism classes of matroids of rank four on eight elements \cite{Matsumotoetal:2012}; our computation is based on the data from \url{http://www-imai.is.s.u-tokyo.ac.jp/~ymatsu/matroid}.
For computing the tropical linear spaces given by the valuations defined above we employed the enriched version available at \url{db.polymake.org}.
This is certainly not enough to provide a global view on the Dressian $\Dr(4,8)$, but it gives us a first glimpse of relevant combinatorial features.
There are 62 different bounded $f$-vectors of such tropical linear spaces, so we cannot list them all. Also, up to combinatorial isomorphism, there are 465 different subdivisions of the hypersimplex induced by these matroids.
As an example, consider the matroid $M := U_{1,2}^{\oplus 4}$; see \cite[Chapter 4.2]{Oxley:2011} for more on direct sums of matroids. The bounded $f$-vector of the tropical linear space induced by $v_M$ is $(14,24,12,1)$. 
In particular, the last two entries already achieve the respective maxima conjectured by Speyer, which are $(20,30,12,1)$. 
Experiments suggests that this is generally true, i.e., if $M = U_{1,2}^{\oplus d}$, then the valuation on $U_{d,2d}$ gives a linear space whose bounded $f$-vector maximizes the last two entries.
Among valuations of the form $v_M$ on $U_{4,8}$, the maximal number of edges is in fact also 24.
However, the maximal number of vertices is 15. This is achieved by the unique matroid with 56 bases and 14 hyperplanes, i.e., flats of rank three. For experts: This is a \emph{sparse paving} matroid, which has the maximal number 16 of \emph{cyclic flats} among all matroid of rank four on eight elements.

\subsection{Further optimization}

Many of the objects considered here, such as polytopes, fans and matroids, exhibit symmetries which are also visible in the corresponding closure systems. It seems desirable, therefore, to exploit this during the computation. For every orbit of a closed set, only one representative would be computed. In a first approach, this could be achieved by considering equivalent sets to be the same in Algorithm \ref{algorithm:closed_sets}: Once, when collecting all minimal closures $\cl(N \cup \{i\})$ and again when checking if $N_i$ is already in the graph.
One could then easily recover the list of all closed sets in the end, though reconstructing the full poset structure (i.e.\ without symmetry) would require significant computational work.

As mentioned in \ref{subsec:performance}, the most expensive part in our computations is a convex hull algorithm for computing the subdivision and the facets of the matroid polytope. It is known that the facets can be described in terms of the combinatorics of the matroid \cite{FeichtnerSturmfels:2005}. 
It is unclear if such a description can be given for the regular subdivision.

\bibliographystyle{alpha}
\bibliography{References}

\end{document}